\documentclass[12pt]{article}

\usepackage[top=0.75in]{geometry}
\usepackage{amsmath,amsthm,amssymb}
\usepackage{enumerate}

\newtheorem{thm}{Theorem}

\newtheorem{cor}[thm]{Corollary}
\newtheorem{conj}[thm]{Conjecture}

\theoremstyle{definition}
\newtheorem*{defi}{Definition}

\def\lc{\left\lceil} 
\def\rc{\right\rceil}

\begin{document}

\title{Rado Numbers of Regular Nonhomogeneous  Equations}
\author{Thotsaporn ``Aek'' Thanatipanonda\\
Mahidol University International College\\ 
Nakhon Pathom, Thailand}
\date{August 28, 2019}

\maketitle
\thispagestyle{empty}
\vspace{-0.25in}

\begin{abstract}
We consider Rado numbers of the regular equations 
$\mathcal{E}(b)$ of the form
\[   c_1x_1+c_2x_2+\dots+ c_{k-1}x_{k-1} = x_k + b,   \]
where $b \in \mathbb{Z}$ and $c_i \in \mathbb{Z}^{+}$ for all $i$.    
We give the upper bounds and the sufficient condition
for the lower bounds for $t$-color Rado numbers 
$r(\mathcal{E}(b);t)$ in terms of $r(\mathcal{E}(0);t)$
for both $b>0$ and $b<0$. We also give
examples where the exact values of 
Rado numbers are obtained from these results.
\end{abstract}

\section{Introduction}
In 1916 Issai Schur \cite{Schur} showed that for any $t$ colors, $t \geq 1$, there is a 
least positive integer $s(t)$ such that for any $t$-coloring on the interval $[1,s(t)]$,
there must be a monochromatic solution to $x+y=z$ where $x,y$ and $z$ are
positions on the interval. This result is part of Ramsey Theory. 
The numbers $s(t)$ are called \textit{Schur numbers}.
For example $s(2) =5$ 
and the longest possible interval that avoids the mono solution to $x+y=z$ is 
$[1,2,2,1]$ (1 represents red color and 2 represents blue color, for example). 
For 3 colors, $s(3)=14$ and one of the longest 
interval that avoids the mono solution 
to $x+y=z$ is $[1, 2, 2, 1, 3, 3, 3, 3, 3, 1, 2, 2, 1]$. It is also known that
$s(4)=45$ and $s(5)=161.$ We call the equation \textit{$t$-regular} 
if $s(t)$ exists for a given $t$ and \textit{regular} 
if $s(t)$ exists for all $t, \;\ t \geq 1$.

Later on, Richard Rado, a Ph.D. student of Schur, generalized Schur's work to a 
linear homogeneous equation $\sum_{i=1}^k c_ix_i =0$
and found the condition for regularity of these equations, \cite{Rado1,Rado2}. 

\begin{thm}[Rado's Single Equation Theorem]
Let $k \geq 2.$ Let $c_i \in \mathbb{Z}-\{0\},
1 \leq i \leq k,$ be constants. Then
\[  \sum_{i=1}^k c_ix_i =0\]
is regular if and only if there exists a nonempty 
set $D \subseteq \{c_i, \;\ 1 \leq i \leq k \}$
such that $\sum_{d \in D} d =0$.
\end{thm}

As with Schur numbers,
for a linear equation $\mathcal{E}$, 
we denote by $r(\mathcal{E};t)$ 
the minimal integers, if it exists, such that any $t$-coloring
of $[1,r(\mathcal{E};t)]$ must admit a monochromatic solution to 
$\mathcal{E}.$ The numbers $r(\mathcal{E};t)$ are called 
\textit{$t$-color Rado numbers for equation $\mathcal{E}$.}

An analog to Rado's Theorem 
which gives the regularity condition 
for a linear non-homogeneous equation is given below.

\begin{thm} \label{9.10}
Let $k \geq 2$ and let $b, c_1, c_2, \dots, c_k$ be nonzero
integers. Let $\mathcal{E}(b)$ be the equation
\[     \sum_{i=1}^k c_ix_i = b,\]
and let $s =  \sum_{i=1}^k c_i.$ Then $\mathcal{E}(b)$ is regular
if and only if one of the following conditions holds:
\begin{enumerate}[(i)]
\item $\dfrac{b}{s} \in \mathbb{Z^{+}};$
\item  $\dfrac{b}{s}$ is a negative integer 
and $\mathcal{E}(0)$ is regular.
\end{enumerate}
\end{thm}

We note that it is possible that an equation does not have
a mono solution for a coloring on $\mathbb{Z^+}.$
For example, the coloring $[1,2,1,2,1,2,\dots]$ 
avoids the mono solution to 
the equation $x+y=2b+1$ for any $b \geq 0.$
Also some equations are $t$-regular but not regular.
For example, $3x+y-z=2$ is 2-regular with 
$r(\mathcal{E}; 2)= 8$
but not regular according to Theorem \ref{9.10}.

In this paper, we partially quantify Theorem \ref{9.10}
by giving Rado numbers to equations $\mathcal{E}(\tilde{b})$ 
of the form
\begin{equation} \label{main}
c_1x_1+c_2x_2+\dots+ c_{k-1}x_{k-1} = x_k +\tilde{b},   
\end{equation}
where  $c_i \in \mathbb{Z}^{+}$ for all $i$
and $\tilde{b}$ satisfies the condition $(i)$ or $(ii)$ of Theorem \ref{9.10}.
The Rado numbers of \eqref{main} will be written in term of 
the Rado numbers of 
the corresponding homogeneous equation, $\mathcal{E}(0).$

In order to distinguish the Rado numbers of the homogeneous equation 
from those of the non-homogeneous one, we denote by
$R_C(t) = R_{[c_1,c_2,\dots,c_{k-1}]}(t)$ 
the Rado number of the homogeneous equation, $\mathcal{E}(0),$ 
with $t$ colors.

\section{Main Results; case $\tilde{b} < 0$}

We consider the Rado numbers of 
\eqref{main} where 
the constant $\tilde{b}$ is negative. 
Theorem \ref{up1} gives the upper bounds 
and Theorem \ref{low1} gives a 
sufficient condition for the lower bounds.

\begin{thm} \label{up1}
Consider equation $\mathcal{E}(\tilde{b}) = \mathcal{E}(-b)$ of the form 
\[
c_1x_1+c_2x_2+\dots+ c_{k-1}x_{k-1} = x_k - b, \;\ \;\ c_i > 0,\, b > 0.
\]
Let $s = \sum_{i=1}^{k-1} c_i-1$. If $s|b$ and 
$\mathcal{E}(0)$ is $t$-regular then 
\[  r(\mathcal{E}(-b); t) \leq  \left(\frac{b}{s}+1\right)\cdot R_C(t) -\frac{b}{s}.\]
\end{thm}

\begin{proof} Assume $s|b$ and $\mathcal{E}(0)$ is $t$-regular.
Let $r =  (\frac{b}{s}+1)\cdot R_C(t) -\frac{b}{s}.$

We now show that there is no good coloring on the interval $[1,r].$

Define an injective map $f$ from $[1,R_C(t)]$ to $[1,r]$ by
\[ f(w) = \left(\frac{b}{s}+1\right)\cdot w -\frac{b}{s}.   \]
Notice that the $k$-tuple $(w_1, w_2, \dots, w_{k-1}, \sum_{i=1}^{k-1}c_iw_i)$ 
of the equation 
\[c_1x_1+c_2x_2+\dots+ c_{k-1}x_{k-1} = x_k\] 
is made to correspond to the $k$-tuple
$ (f(w_1), f(w_2), \dots, f(w_{k-1}), f(\sum_{i=1}^{k-1}c_iw_i))$ 
in \[c_1x_1+c_2x_2+\dots+ c_{k-1}x_{k-1} = x_k - b.\]
Please check it yourself!!!

Now given any coloring $\alpha$ on $[1,r]$, 
we define the coloring $\chi$ on the interval $[1,R_C(t)]$ by 
\[   \chi(w) := \alpha( f(w) )  , \;\ \;\ w = 1,2,\dots, R_C(t).   \]
From the definition of the Rado number, any coloring on $[1,R_C(t)]$ must 
contain a mono tuple to $\mathcal{E}(0)$.
Hence there is also a mono 
tuple on $[1, r]$ to $\mathcal{E}(-b)$.
\end{proof}


Next we define a sufficient condition for the lower bounds. 

\begin{defi}[excellence condition]
The coloring on an interval $[1,n]$ satisfies an excellence condition if
it does not contain any mono solution to
\[c_1x_1+c_2x_2+\dots+ c_{k-1}x_{k-1}+j = x_k,\]
for each $j, \;\ 0 \leq j \leq s=  \sum_{i=1}^{k-1} c_i-1$.
\end{defi}

\begin{thm} \label{low1}
Consider the equation  $\mathcal{E}(\tilde{b})=\mathcal{E}(-b)$ of the form
\begin{equation} \label{negb}
c_1x_1+c_2x_2+\dots+ c_{k-1}x_{k-1} = x_k - b, \;\ \;\ 
\text{ where } c_i > 0,\, b > 0.
\end{equation} 
Let $s = \sum_{i=1}^{k-1} c_i-1$. If $s|b$ 
and there is a coloring on the interval $[1,n]$ 
which satisfies an excellence condition then 
\[  r(\mathcal{E}(-b); t) \geq \left(\dfrac{b}{s} +1\right)\cdot n +1.\]
\end{thm}

\begin{proof}
Assume $s|b$ and let $\chi$ be the coloring on $[1, n]$ that satisfies 
an excellence condition to the equation
\[c_1x_1+c_2x_2+\dots+ c_{k-1}x_{k-1} + j = x_k,  \;\ \;\ 0 \leq j \leq s.\]
Let $r= (\frac{b}{s}+1)\cdot n +1.$

We show that there is a ``good coloring'' to $\mathcal{E}(-b)$
on the interval $[1, r-1] = [ 1, \left(\frac{b}{s}+1\right)\cdot n ].$

We define the coloring $\alpha$ on  
$[1, \left(\frac{b}{s}+1\right)\cdot n ]$ by
\[ \alpha(i) = \chi \left( \lc \frac{i}{\frac{b}{s}+1}  \rc    \right) .  \]
Basically, we create the coloring by repeating each point 
of the original coloring on the $[1,n]$ interval $\frac{b}{s}+1$ times. 
We now prove the statement by contradiction: 

Assume there is a mono $k$-tuple on 
$[1, \left(\frac{b}{s}+1\right)\cdot n ]$ to equation \eqref{negb}
written in the form
\[     \left(d_1 \left(\frac{b}{s}+1\right)-e_1,  d_2 \left(\frac{b}{s}+1\right)-e_2, \dots,  
d_{k-1} \left(\frac{b}{s}+1\right)-e_{k-1}, 
 \left(\frac{b}{s}+1\right)\cdot \sum_{i=1}^{k-1}c_id_i -\sum_{i=1}^{k-1} c_ie_i +b \right) ,    \]
where  $ 1 \leq d_i \leq n $  for all $i$ and $0 \leq e_i \leq b/s.$

Notice that $\alpha( d_i(\frac{b}{s}+1)-e_i) = \chi(d_i).$
However, by this mapping, we have the mono $k$-tuple
in $\chi$ as
\[  \left( d_1, d_2,  \dots, d_{k-1},  \sum_{i=1}^{k-1} c_id_i  
+ \lc \frac{b-\sum_{i=1}^{k-1} c_ie_i}{\dfrac{b}{s}+1}  \rc  \right)    \] 

But this is a mono solution to
\[ c_1x_1+c_2x_2+\dots+ c_{k-1}x_{k-1} + j= x_k, \] 
for some $j, \;\ 0 \leq  j  \leq   \lc \frac{sb}{b+s}\rc$
which contradicts the excellence condition of $\chi$
we assumed it to have.  
\end{proof}

We note that the upper bounds and lower bounds meet
if there is a good coloring of length $n=R_C(t)-1$ that satisfies
the excellence condition.

\begin{cor} \label{cor1}
Consider the equation $\mathcal{E}(-b),$
\[  x_1+x_2+\dots+ x_{k-1} = x_k- b , \;\ \;\ \text{with } k \geq 2, 
\;\ b >0 \text{ and } (k-2)|b.   \]
We let $m= b/(k-2).$ Then
\[ r(\mathcal{E}(-b);2) = (m+1)(k^2-k-2)+1. \]
\end{cor}

\begin{proof}
It is known (i.e. Theorem 8.23 of \cite{LR}) that 
\[ r(x_1+x_2+\dots+ x_{k-1} = x_k ;2) = k^2-k-1,  \;\ \;\ \text{ for } k\geq 2. \]
The coloring $\chi = [1^{k-2},2^{(k-1)(k-2)},1^{k-2}]$ 
satisfies the excellence condition for each $k$.
The result follows from Theorems \ref{up1} and \ref{low1}. 
\end{proof}

This result agrees with Theorems 9.14 and 9.26 
of \cite{LR} which applies to any 2-coloring 
but for a more general $b$ (not only $(k-2)|b$).
However, our result applies to any $t$-coloring.

\begin{cor}
For $m > 0,$
\begin{align*}
 r(x+y-z = -m ;3) &= 13m+14, \\
 r(x+y+z-w = -2m ;3) &= 42m+43, \\
 r(x_1+x_2+x_3+x_4-x_5 = -3m ;3) &= 93m+94, \\ 
 r(x_1+x_2+x_3+x_4+x_5-x_6 = -4m ;3) &= 172m+173. 
\end{align*}
\end{cor}
The first result was also mentioned in \cite{RS} and \cite{Schaal2}.
The good colorings (that also satisfy the excellence condition) of the
first two equations can be found by the 
accompanying program \texttt{Schaal}.
The good colorings (that also satisfy the excellence condition)
of the equations $x_1+x_2+x_3+x_4=x_5$
and $x_1+x_2+x_3+x_4+x_5=x_6$
were given in \cite{Schaal}. 

\section{Main Results; case $\tilde{b} > 0$}

We consider the Rado numbers of 
\eqref{main} where 
the constant $\tilde{b}$ is positive. 
Theorem \ref{up2} gives the upper bounds 
and Theorem \ref{low2} gives a 
sufficient condition for the lower bounds.
 
\begin{thm} \label{up2}
Consider the equation  $\mathcal{E}(\tilde{b})=\mathcal{E}(b)$ of the form
\[
c_1x_1+c_2x_2+\dots+ c_{k-1}x_{k-1} = x_k + b, \;\ \;\ 
\text{ where } c_i > 0, \;\ b > 0.
\]
Let $s = \sum_{i=1}^{k-1} c_i-1$. If $s|b$ 
and $\mathcal{E}(0)$ is $t$-regular then 
\[  r(\mathcal{E}(b); t) \leq \dfrac{b}{s}-\lc\dfrac{b}{s \cdot R_C(t)}\rc+1.\]
\end{thm}

\begin{proof} Assume $s|b$ and $\mathcal{E}(0)$ is $t$-regular.
Let $r = \dfrac{b}{s}-\lc\dfrac{b}{s \cdot R_C(t)}\rc+1.$

We write $b$ as
$b = s\left( R_C(t) \cdot m - q \right)$ 
where $ m = \lc \dfrac{b}{s\cdot R_C(t)} \rc$ 
and $0 \leq q \leq R_C(t)-1$.
Then $r = (R_C(t)-1)\cdot m -q +1$.

We show that there is no good coloring on the interval $[1,r].$

\textbf{Case 1:} $m=1$.\\
Then $r=b/s.$ We have a trivial mono solution to $\mathcal{E}(b)$ 
via $x_1=x_2=x_3=\dots=x_k=r.$

\textbf{Case 2:} $m>1$.\\
Define an injective map $f$ from $[1,R_C(t)]$ to $[1,r]$ by
 \[ f(w) = (R_C(t)-w)\cdot m -q +w.   \]
Notice that a tuple $(w_1, w_2, \dots, w_{k-1}, \sum_{i=1}^{k-1}c_iw_i)$ 
of the equation 
\[c_1x_1+c_2x_2+\dots+ c_{k-1}x_{k-1} = x_k\] 
is made to correspond to the tuple
$ (f(w_1), f(w_2), \dots, f(w_{k-1}), f(\sum_{i=1}^{k-1}c_iw_i))$ 
in \[c_1x_1+c_2x_2+\dots+ c_{k-1}x_{k-1} = x_k + b.\]
Please check it yourself!!!

Now, given any coloring of $\alpha$ on $[1,r]$, 
we define the coloring $\chi$ on the interval $[1,R_C(t)]$ by 
\[   \chi(w) = \alpha( f(w) )  , \;\ \;\ w = 1,2,\dots, R_C(t).   \]
From the definition of the Rado number, any coloring on $[1,R_C(t)]$ must 
contain a mono tuple to $\mathcal{E}(0)$.
Hence there is also a mono 
tuple on $[1, r]$ to $\mathcal{E}(b)$.
In both cases, there is no good coloring on $[1,r].$
\end{proof}


The lower bounds  
can be stated in similar way to Theorem \ref{low1}.

\begin{thm} \label{low2}
Consider the equation  $\mathcal{E}(\tilde{b})=\mathcal{E}(b)$ of the form
\begin{equation} \label{posb}
c_1x_1+c_2x_2+\dots+ c_{k-1}x_{k-1} = x_k + b, \;\ \;\ 
\text{ where } c_i > 0, \;\ b > 0.
\end{equation} 
Let $s = \sum_{i=1}^{k-1} c_i-1$. If $s|b$ 
and there is a coloring on the interval $[1,n]$ 
which satisfies the excellence condition then 
\[  r(\mathcal{E}(b); k) \geq \dfrac{b}{s}-\lc\dfrac{b}{s \cdot (n+1)}\rc+1.\]
\end{thm}

\begin{proof} We invoke 
the result of Theorem \ref{low1} by rewriting \eqref{posb}
in the form of \eqref{negb}. 
Since $s|b,$ we can write $b$ in the form 
$b =s\left[ (n+1)m-q \right]$ 
where $m=\lc \dfrac{b}{s \cdot (n+1)} \rc$
and $0 \leq q \leq n.$ Then
$r= \dfrac{b}{s}-\lc\dfrac{b}{s \cdot (n+1)}\rc+1
= (n+1)m-q-m+1 = nm-q+1.$

We show that there is a ``good coloring'' on the interval 
$[1, r-1] = [ 1,  nm-q ]$ to \eqref{posb}.

First we rewrite \eqref{posb} as
\[ c_1x_1+c_2x_2+\dots+ c_{k-1}x_{k-1} = x_k + s\left[ (n+1)m-q \right]. \]
We then rewrite this equation again as
\begin{align*}
& c_1\left[(n+1)m-q-x_1\right]+c_2\left[(n+1)m-q-x_2\right]+\dots
+ c_{k-1}\left[(n+1)m-q-x_{k-1}\right] \\
&= (n+1)m-q-x_k. 
\end{align*}
Next we add $-s(m-1)$ on both sides of the equation,
\begin{align*}
& c_1\left[nm-q+1-x_1\right]+c_2\left[nm-q+1-x_2\right]+\dots
+ c_{k-1}\left[nm-q+1-x_{k-1}\right] \\ 
&= [nm-q+1-x_k ]-s(m-1). 
\end{align*}
We let $x'_i = nm-q+1-x_i$ for each $i.$ 
The reader sees that 
$x'_i$ is $x_i$ after reversing the interval $[1, nm-q].$  
The equation after substitution is
\begin{equation} \label{last} 
c_1x'_1+c_2x'_2+\dots+ c_{k-1}x'_{k-1} = x'_k -s(m-1). 
\end{equation}
The next step is clear. We invoke the result from Theorem \ref{low1} (!) that
there is a good coloring $\alpha$ on the interval $[1,mn]$ to \eqref{last}. 
We can then make a good coloring to \eqref{posb} from this interval by taking the
elements $1$ to $mn-q$ of $\alpha$ and reverse the interval.
\end{proof}

Below are some applications of Theorems \ref{up2} and \ref{low2}.

\begin{cor} Consider the equation $\mathcal{E}(b)$ of the form
\[  x_1+x_2+\dots+ x_{k-1} = x_k + b , \;\ \;\ \text{with } k \geq 2, 
\;\ b \geq 1 \text{ and } (k-2)|b.   \]
We let $m= b/(k-2).$ Then
\[ r(\mathcal{E}(b);2) = m- \lc \dfrac{m}{k^2-k-1} \rc+1. \]
\end{cor}

\begin{proof}
The proof is the same as for Corollary \ref{cor1} 
except that this time we apply Theorems 
\ref{up2} and \ref{low2}. 
\end{proof}
Note that the above result when $k=3$ was mentioned in \cite{BL}.

\begin{cor}
For $b \geq 1$, 
\[ r(x+y-z=b; 3) = b-\lc\frac{b}{14} \rc +1. \]
\end{cor}

\begin{proof}
The proof is straight forward. It can be checked that the original coloring
$[1, 2, 2, 1, 3, 3, 3, 3, 3, 1, 2, 2, 1]$ satisfies
the excellence condition. Then we apply Theorems
\ref{up2} and \ref{low2}.
\end{proof}

This result was a part of Theorem 9.15 in \cite{LR}. 
Although it was wrongly claimed that
$ r(x+y-z=b; 3) = b-\lc\frac{b-1}{14} \rc. $

For the situation when the equation $\mathcal{E}(0)$ given by
\[
c_1x_1+c_2x_2+\dots+ c_{k-1}x_{k-1} = x_k, \;\ \;\ c_i > 0,\] 
is not $t$-regular, the trivial bounds of the Rado numbers to
\[
c_1x_1+c_2x_2+\dots+ c_{k-1}x_{k-1} = x_k+b, \;\ \;\ b > 0, \, s|b,\] 
are
\[ \lc \dfrac{b+1}{s+1} \rc \leq r(\mathcal{E}(b); t) \leq \dfrac{b}{s}, \;\ \text{ for any } t \geq 1.\]
The mono solution for the upper bound arises from the tuple 
$(\frac{b}{s},\frac{b}{s},\dots,\frac{b}{s})$.


\section{Final Remarks}
So far, our results were obtained by checking 
the excellence condition of each good coloring.
For the 2-coloring and 3-coloring,
it seems that there are always colorings 
of length $n = R_C(t)-1$ to the equations 
\[  c_1x_1+c_2x_2+\dots+ c_{k-1}x_{k-1} = x_k, \;\ \;\  \text{ where } c_i > 0 \]
that satisfy the excellence condition. 
Thus it makes sense to make the following conjecture.

\begin{conj} For $t =2$ or $3$, fix constants $c_1, c_2, \dots, c_{k-1}.$
Consider the equation  $\mathcal{E}(\tilde{b})$ of the form
\[
c_1x_1+c_2x_2+\dots+ c_{k-1}x_{k-1} = x_k +\tilde{b}, \;\ \;\ 
\text{ where } c_i > 0.
\]
Let $s = \sum_{i=1}^{k-1} c_i-1$. If $s|\tilde{b}$ 
and $\mathcal{E}(0)$ is $t$-regular then 
\[  r(\mathcal{E}(\tilde{b}); t) = 
\begin{cases} 
\dfrac{\tilde{b}}{s}-\lc\dfrac{\tilde{b}}{s \cdot R_C(t)}\rc+1,  & \text{ for } \tilde{b} >0, \\
 -\dfrac{\tilde{b}}{s}\cdot (R_C(t)-1) +R_C(t)     , & \text{for } \tilde{b}< 0.
\end{cases}\]
\end{conj}

For $t$-colorings where $t \geq 4$,
our Maple program is too 
slow to give any tangible observations. 
A faster program could used to verify 
whether this conjecture still holds.

Lastly, the reader might wonder about the other type of equations that
we did not consider, i.e. 
\[  \sum_{i=1}^{k-1}c_ix_i= c_kx_k+b, \;\ \;\ 
\text{where }  c_i \geq 1 , \;\ \text{ for } 1 \leq i \leq k-1 \text{ and } c_k \geq 2. \]
It turns out that the Rado numbers of these equations 
exhibit more complicated patterns from those discovered in this paper. 


\end{document}